\documentclass[12pt,reqno]{amsart}
\usepackage[margin=1in]{geometry}

\newcommand\EatDot[1]{}

\usepackage{amsfonts,amsmath,amsthm,amssymb} 
\usepackage{mathrsfs} 
\usepackage{stmaryrd} 
\usepackage{mathtools}

\usepackage{natbib}
\usepackage{url}

\usepackage{color} 

\usepackage{multirow}

\allowdisplaybreaks[1]

\newtheorem{theorem}{Theorem}
\newtheorem{proposition}[theorem]{Proposition}
\newtheorem{lemma}[theorem]{Lemma}

\newtheorem*{question}{Q}

\theoremstyle{remark}
\newtheorem{remark}{Remark}

\theoremstyle{definition}
\newtheorem{definition}{Definition}
\newtheorem{assumption}{Assumption}

\usepackage[justification=centering]{caption}

\newcommand{\bseta}{\boldsymbol{\bseta}}

\newcommand{\bomega}{\boldsymbol{\bomega}}

\newcommand{\bbS}{\mathbb{S}}

\newcommand{\calB}{\mathcal{B}}
\newcommand{\calC}{\mathcal{C}}

\newcommand{\calN}{\mathcal{N}}
\newcommand{\calO}{\mathcal{O}}

\newcommand{\N}{\mathbb{N}}

\newcommand{\R}{\mathbb{R}}

\newcommand{\E}{\mathbb{E}}
\renewcommand{\Pr}{\mathbb{P}}
\newcommand{\Cov}{\textnormal{Cov}}
\newcommand{\iid}{\textnormal{i.i.d.}}

\newcommand{\diff}{\mathrm{d}}

\renewcommand{\epsilon}{\varepsilon}

\DeclareMathOperator*{\argmin}{arg\,min}

\newcommand{\ctr}{\textnormal{ctr}}

\newcommand{\tx}{\tilde{x}}
\newcommand{\tX}{\tilde{X}}
\usepackage[UKenglish]{babel}

\title[Covariance estimation via ULA]{On sample complexity for covariance estimation via the unadjusted Langevin algorithm}
\author[S Nakakita]{Shogo Nakakita}
\address{Komaba Institute for Science, University of Tokyo, 3-8-1, Komaba, Meguro, Tokyo 153-8902, Japan}
\email{nakakita@g.ecc.u-tokyo.ac.jp}
\begin{document}

\begin{abstract}
    We establish sample complexity guarantees for estimating the covariance matrix of a strongly log-concave smooth distribution using the unadjusted Langevin algorithm (ULA).
    We quantitatively compare our complexity estimates on single-chain ULA with embarrassingly parallel ULA and derive that the sample complexity of the single-chain approach is smaller than that of embarrassingly parallel ULA by a logarithmic factor in the dimension and the reciprocal of the prescribed precision, with the difference arising from effective bias reduction through burn-in.
    The key technical contribution is a concentration bound for the sample covariance matrix around its expectation, derived via a log-Sobolev inequality for the joint distribution of ULA iterates.
\end{abstract}

\maketitle

\section{Introduction}\label{sec:intro}
Sampling from a Gibbs distribution $\pi(\diff x)\propto e^{-f(x)}\diff x$ on $\R^{d}$ is a fundamental problem in statistics, physics, and machine learning.
Let $f\in\calC^{2}(\R^{d};\R)$ be an $\alpha$-convex and $\beta$-smooth potential function, i.e., $\alpha I_{d}\preceq\nabla^{2}f(x)\preceq\beta I_{d}$ for all $x\in\R^{d}$.
Markov Chain Monte Carlo (MCMC) is one of the successful strategies for this problem, and the unadjusted Langevin algorithm (ULA) is a popular MCMC method due to its simplicity and efficiency.

Let us introduce the definition of ULA (with fixed step size). Given (i) an initial value $X_{0}$ (a $d$-dimensional random vector, possibly $X_{0}\sim \delta_{x}$ for some $x\in\R^{d}$), (ii) step size $\eta>0$, (iii) the number of burn-in steps $m\ge0$, and (iv) sample size $n\ge1$, the algorithm is defined for all $i=1,\ldots,m+n$ as
\begin{equation}\label{eq:def:ula}
    X_{i}=X_{i-1}-\eta\nabla f(X_{i-1})+\sqrt{2\eta}Z_{i},
\end{equation}
where $\{Z_{i}\}_{i\in\N}\sim^{\iid} \calN_{d}(0,I_{d})$ independent of $X_{0}$.
While its invariant distribution $\pi_{\eta}$ (which exists for sufficiently small $\eta$) differs from the target distribution $\pi$, it is known that $\pi_{\eta}$ converges to $\pi$ in some metrics or divergences as $\eta\to0$.

ULA is derived by the Euler--Maruyama discretization of the Langevin dynamics (also known as the overdamped Langevin diffusion) defined as the solution of the following stochastic differential equation:
\begin{equation*}
    \diff Y_{t}=-\nabla f(Y_{t})\diff t+\sqrt{2}\diff B_{t},\ Y_{0}=X_{0},
\end{equation*}
where $B_{t}$ is the standard $d$-dimensional Brownian motion.
We can interpret that ULA approximates the continuous-time process $Y_{t}$ at discrete time points $t=i\eta$ for $i=0,\ldots,m+n$.
The Langevin dynamics has $\pi$ as its invariant distribution and has an explicit rate of convergence to $\pi$ if $\pi$ satisfies functional inequalities such as Poincar\'{e} inequalities and log-Sobolev inequalities \citep{bakry2014analysis}.
Owing to this property, ULA is expected to approximate $\pi$ well if $\eta\ll1$ and $m\eta\gg1$.
In fact, the community has unveiled the behaviour of the distribution of $X_{n}$ for each $n$ \citep[e.g.,][]{dalalyan2017theoretical,durmus2017nonasymptotic,durmus2019high,vempala2019rapid} and found favourable non-asymptotic convergence guarantees to the target distribution $\pi$ in various metrics such as the 2-Wasserstein distance and Kullback--Leibler divergence in this decade.

\subsection{Problem setting}\label{sec:intro:setting}
Our main interest lies in the sample complexity analysis for moment estimation via ULA.
In statistical inference, we are often interested in moments of target distributions, such as means and covariance matrices, and the sample complexities required for their estimation. 
We expect that the sample mean $\bar{X}_{n}$ and sample covariance matrix $\hat{\Sigma}_{n}$ defined as
\begin{equation}\label{eq:def:samplemoments}
    \bar{X}_{n}=\frac{1}{n}\sum_{i=m+1}^{m+n}X_{i},\quad \hat{\Sigma}_{n}=\frac{1}{n}\sum_{i=m+1}^{m+n}X_{i}X_{i}^{\top}-\bar{X}_{n}\bar{X}_{n}^{\top}
\end{equation}
of ULA should approximate the population mean $\E_{\pi}[X]=\int x\pi(\diff x)$ and population covariance $\Cov(\pi)=\int xx^{\top}\pi(\diff x)-\E_{\pi}[X]\E_{\pi}[X]^{\top}$ when $\eta\ll1$ and $n\eta\gg1$.
While this idea is asymptotically guaranteed by standard ergodic theorems under fixed $d$, it is not obvious when $d$ is large.
Specifically, it is unclear how large $n$ should be to achieve a certain accuracy in high-dimensional settings due to the dependence among the ULA outputs and the resulting large variance.

One may consider the embarrassingly parallel ULA to address dependence, which is the main comparison target of this paper.
In this study, we refer to the embarrassingly parallel ULA as the procedure that collects the last iterates of $N$ independent paths of ULA $\{X_{m+1}^{k}\}_{k=1}^{N}$ defined as follows: for all $k=1,\ldots,N$ and $i=1,\ldots,m+1$ (with $m$ burn-in steps discarded in estimation), with $d$-dimensional random vectors $X_{0}^{k}$,
\begin{equation}\label{eq:def:ep}
    X_{i}^{k}=X_{i-1}^{k}-\eta\nabla f(X_{i-1}^{k})+\sqrt{2\eta}Z_{i}^{k},
\end{equation}
where $\{Z_{i}\}_{i,k}\sim^{\iid}\calN_{d}(0,I_{d})$.
While it requires burn-in for every path to remove \emph{bias}, the \emph{variance} of its sample moments is smaller than that of single-chain ULA with $N=n$.
If the sample complexities of both methods coincide up to constant factors, embarrassingly parallel ULA is preferable in practice due to its wall-clock speedup via parallelization.
Hence, one should raise the following question:
\begin{question}
    Does the bias removal of single-chain ULA by burn-in outweigh its large variance from dependence in comparison to embarrassingly parallel ULA in terms of sample complexity? If so, how large is the difference?
\end{question}

To answer this question, we study sample complexity guarantees for covariance estimation of strongly log-concave smooth distributions via single-chain ULA and embarrassingly parallel ULA respectively.
In particular, this study establishes a concentration bound for the sample covariance matrix of single-chain ULA, which is our main theoretical contribution.
Some studies have investigated non-asymptotic guarantees for moment estimation via MCMC \citep{durmus2019high,kook2024covariance,shen2025high}, but this topic remains largely unexplored in the literature.
Recently, \citet{kook2024covariance} studied covariance matrix estimation with high-accuracy MCMC (that is, MCMC whose invariant distribution is exactly the target distribution; for example, random-walk Metropolis and Metropolis-adjusted Langevin algorithm) for $\pi$ satisfying a Poincar\'{e} inequality.
However, their comparison focuses on query complexities rather than sample complexities for covariance estimation.
In contrast, this work examines sample complexity through a concentration analysis of the sample covariance matrix and establishes explicit bounds for both single-chain and parallel ULA approaches.

We address the problem via a bias--variance decomposition and analysing the variance part; specifically, we show concentration bounds on sample moments of ULA around its expectation rather than the moments of $\pi_{\eta}$ or $\pi$.
By the triangular inequality with operator norm $\|\cdot\|$, we have
\begin{equation}\label{eq:bvd}
    \left\|\hat{\Sigma}_{n}-\Cov(\pi)\right\|
    \le \underbrace{\left\|\Cov(\pi_{\eta})-\Cov(\pi)\right\|}_{\text{bias by discretization}}
    +\underbrace{\left\|\E\left[\hat{\Sigma}_{n}\right]-\Cov(\pi_{\eta})\right\|}_{\text{bias by non-stationarity}}
    +\underbrace{\left\|\hat{\Sigma}_{n}-\E\left[\hat{\Sigma}_{n}\right]\right\|}_{\text{variance by sampling}},
\end{equation}
where $\Cov(\pi_{\eta}):=\int_{\R^{d}}xx^{\top}\pi_{\eta}(\diff x)-\int_{\R^{d}}x\pi_{\eta}(\diff x)(\int_{\R^{d}}x\pi_{\eta}(\diff x))^{\top}$.
Since previous studies \citep[e.g.,][]{durmus2019high} show bounds on the 2-Wasserstein distances for the pair of $\pi_{\eta}$ and $\pi$ and that of the distribution of $X_{i}$ and $\pi_{\eta}$, our focus is on evaluating the variance term.

\subsection{Main contributions}\label{sec:intro:contrib}

Our contributions can be summarized as follows.
\begin{itemize}
    \item \textbf{Concentration for sample covariance matrices under dependence.}
    Using the log-Sobolev inequality, we derive a concentration bound on the sample covariance matrix around its expectation in operator norm, whose rate of convergence is $(1/\alpha)\sqrt{d/(\alpha \eta n)}$ (Theorem~\ref{thm:concentration}).
    This rate is reasonable since $1/\alpha$ is the largest eigenvalue of the covariance matrix under Gaussian settings, and $\alpha \eta n$ can be regarded as the terminal time of the Langevin dynamics which ULA approximates and should serve as the effective sample size in an ordinary sense.
    \item \textbf{Sample complexity for covariance estimation.}
    Combining the above variance bounds with existing non-asymptotic bias bounds for ULA, we obtain explicit sample complexity guarantees for estimating the covariance matrix of the target distribution $\pi$ via a bias--variance decomposition.
    Under standard settings, the total sample complexity is shown to be of order $\calO(\beta^{2}d^{2}(d+\log(1/\delta))\alpha^{-6}\epsilon^{-4})$ to achieve $\epsilon$-accuracy in operator norm with probability at least $1-\delta$ (Proposition~\ref{prop:comp:ula}).
    Notably, the estimation error does not deteriorate for large $n$ beyond the required sample complexity, which is a desirable property in statistics and optimization.

    \item \textbf{Quantitative comparison of sample complexity between ULA and embarrassingly parallel ULA.}
    We also derive sample complexity guarantees for covariance estimation via embarrassingly parallel ULA by the $N$ last iterates of independent path realizations of ULA as a comparison.
    The estimate on the total sample complexity of embarrassingly parallel ULA is shown to be larger than that of ULA by a \emph{logarithmic factor in $d$ and $\epsilon$}.
    While it is the comparison of upper bounds, both the derivations employ tight concentration inequalities via the isoperimetry argument and same known bounds on moments and Wasserstein distances.
    In addition, the difference can be attributed to the bias reduction mechanism by burn-in.
    The comparison quantifies the difference between single-chain ULA and embarrassingly parallel ULA in sample complexity for covariance estimation.
\end{itemize}
To derive these results, we establish a log-Sobolev inequality for the joint distribution of $(X_{m+1},\ldots,X_{m+n})$, which is our main technical contribution.
In particular, the joint distribution is shown to satisfy a log-Sobolev inequality with constant $2/(\alpha^{2}\eta)$ under a mild burn-in condition (Proposition~\ref{prop:lsi}); this constant is shown to be tight by exemplifying a Gaussian AR(1) process.

\subsection{Literature review}\label{sec:intro:literature}
We present literature review on empirical processes of MCMC, invariant distributions of ULA, and concentration under dependence.

\subsubsection*{Non-asymptotic analysis for empirical processes of MCMC}
We first review recent studies about non-asymptotic bounds on empirical processes of MCMC algorithms.
\citet[Section 4]{durmus2019high} consider Lipschitz concentration of ULA under the strong log-concavity and smoothness of target distributions via iterative moment bounds.
\citet{kook2024covariance} study concentration bounds for sample covariance matrices of high-accuracy MCMC (i.e., MCMC having invariant measures coincide with target measures and resulting polylogarithmic complexity in error tolerance level under warm-start) and query complexities for MCMC and i.i.d.~sampling; though our interest is in sample complexities, the interest of their research is close to ours.
\citet{shen2025high} derive high-dimensional normal approximations for the difference between the sample mean of stationary ULA and expectation of its target distribution.

\subsubsection*{Properties of invariant distributions of ULA}
Let us also introduce analysis on the invariant measure $\pi_{\eta}$ of ULA with stepsize $\eta>0$.
\citet[arXiv v4]{vempala2019rapid} show that $\pi_{\eta}$ satisfies a log-Sobolev inequality with constant $C\le 2/\alpha$.
\citet[arXiv v2]{altschuler2022concentration} reveal subgaussianity and subexponentiality of $\pi_{\eta}$ under the strong log-concavity and log-concavity of target distributions respectively.

\subsubsection*{Concentration under dependence}
Moment bounds and concentration bounds under dependence are also a classical but topical problem \citep[e.g.,][]{han2020moment,neeman2024concentration}.
While there are several strategies to derive sharp bounds as independent cases, one hopeful approach is to employ the isoperimetry of joint distributions, which is also pointed out in \citet{kook2024covariance}.
\citet{adamczak2015note} studies the Hanson--Wright inequality without assuming independent elements via the isoperimetry approach.
\citet{nakakita2025corrigendum} also employ this approach, assuming a log-Sobolev inequality of the joint distribution of dependent random vectors.
Our analysis follows this route: we establish a sharp, dimension-free log-Sobolev constant for the joint law of $(X_{m+1},\ldots,X_{m+n})$ generated by ULA, which leads to a concentration bound for the sample covariance matrix in operator norm.

\subsection{Notation}\label{sec:intro:notation}
For a $d$-dimensional vector $x\in\R^{d}$, let $|x|$ be the Euclidean norm of $x$.
For a matrix $A\in\R^{d_{1}\times d_{2}}$, let $A^{\top}$ be the transpose of $A$.
$\nabla$ denotes the gradient operator; when needed, it is understood in the weak sense.
For a symmetric matrix $A\in\R^{d\times d}$, let $\|A\|$ be the operator norm of $A$.
For two symmetric matrices $A,B\in\R^{d\times d}$, we write $A\preceq B$ if $B-A$ is positive semi-definite.
For a $d$-dimensional vector $x\in\R^{d}$, let $\delta_{x}(\cdot)=\delta_{0}(\cdot-x)$, where $\delta_{0}$ is the Dirac delta.
For a pair of probability distributions $\nu_{1}$ and $\nu_{2}$ with finite second moments, let $W_{2}(\nu_{1},\nu_{2})$ be the 2-Wasserstein distance between $\nu_{1}$ and $\nu_{2}$.

\subsection{Paper organization}\label{sec:intro:organization}
Section \ref{sec:concentration} presents our theoretical results for uniform concentration.
In Section \ref{sec:appl}, we apply the results to sample complexity analysis for the estimation of the covariance matrix of the target distribution $\pi$.
Appendix \ref{app:proof} gives the proofs of the theoretical results.

\section{Concentration on the sample covariance matrix of the unadjusted Langevin algorithm}\label{sec:concentration}
In this section, we present our theorem on the concentration of the sample covariance matrix \eqref{eq:def:samplemoments} of ULA around its expectation.
Furthermore, we discuss an extension to the projected ULA.

\subsection{Setup}\label{sec:concentration:setup}
We give definitions and assumptions used in this section and the rest of the paper.

\begin{definition}
    A probability distribution $\nu$ on $\R^{d}$ is said to satisfy a log-Sobolev inequality with constant $C\ge 0$ (denoted by LSI($C$)) if for any smooth function $g:\R^{d}\to\R$ with $\int_{\R^{d}}g^{2}(x)|\log g(x)|\mu(\diff x)<\infty$,
    \begin{equation*}
        \int_{\R^{d}}g^{2}(x)\log g^{2}(x)\nu(\diff x)-\left(\int_{\R^{d}}g^{2}(x)\nu(\diff x)\right)\log\left(\int_{\R^{d}}g^{2}(x)\nu(\diff x)\right)\le 2C\int_{\R^{d}}|\nabla g(x)|^{2}\nu(\diff x).
    \end{equation*}
\end{definition}
See \citet{bakry2014analysis}, \citet{vanhandel2016probability}, and \citet{chewi2023log} for details on log-Sobolev inequalities and their properties.
For example, the $d$-dimensional standard Gaussian distribution $\calN_{d}(0,I_{d})$ satisfies an LSI(1).

The following assumption is standard in the analysis of ULA \citep[e.g.,][]{durmus2019high}.
\begin{assumption}\label{asmp:potential}
    The potential function $f\in\calC^{2}(\R^{d};\R)$ satisfies $\alpha I_{d}\preceq \nabla^{2}f \preceq \beta I_{d}$ with some $\alpha,\beta>0$ ($\alpha\le \beta$).
\end{assumption}

\subsection{Concentration bound}\label{sec:concentration:main}
We now present our theorem on the concentration of the sample covariance matrix of ULA.
The proof is deferred to Appendix \ref{app:proof}.
\begin{theorem}\label{thm:concentration}
    Suppose that Assumption~\ref{asmp:potential} holds.
    Assume that $\eta\in(0,1/\beta)$ and the distribution of $X_{0}$ satisfies a log-Sobolev inequality with constant $\kappa_{0}$.
    If 
    \begin{equation*}
        m\ge 0\vee\left(\frac{\log(\kappa_{0}/\eta)}{\log(1/(1-\alpha\eta))}-1\right)\vee \left(\frac{\log(4(\E[|X_{0}-x^{\ast}|^{2}]+d/\alpha)/\eta n)}{\log(1/(1-\alpha\eta))}-1\right)
    \end{equation*}
    (with convention $\log0=-\infty$), then for any $\delta>0$, with probability at least $1-4\delta$,
    \begin{align*}
        \left\|\hat{\Sigma}_{n}-\E\left[\hat{\Sigma}_{n}\right]\right\|\le \frac{16}{\alpha}\left(\sqrt{\frac{2(9d+4\log\delta^{-1})}{\alpha\eta n}}\vee \left(\frac{9 d+4\log\delta^{-1}}{\alpha\eta n}\right)\right).
    \end{align*}
\end{theorem}
Note that selecting $x=x^{\ast}=\argmin_{x}f(x)$ benefits the reduction of the bias terms in \eqref{eq:bvd} in our argument employing \citet{durmus2019high}.
Solving a optimization problem for strongly convex smooth $f$ is computationally easy (e.g., the gradient descent shows linear convergence), and its complexity is negligible in comparison to that of the sampling problem as \citet{altschuler2024faster} point out.
Let us remark that the rate of convergence $(1/\alpha)\sqrt{d/(\alpha\eta n)}$ is reasonable: the scaling $\alpha$ is natural since $\|\Sigma\|=1/\alpha$ for a Gaussian case $f(x)=x^{\top}\Sigma^{-1}x/2$ with a positive definite matrix $\Sigma$, and the effective sample size $\alpha n\eta$ is standard from the Langevin dynamics as a continuous analogue of ULA, whose effective sample size is $\alpha \eta n$ \citep[e.g., see][]{bakry2014analysis}.

\subsection{Technical remarks on Theorem~\ref{thm:concentration}}\label{sec:concentration:remark}
The technical core of Theorem~\ref{thm:concentration} is to show that the joint distribution of $(X_{m+1},\ldots,X_{m+n})$ generated by ULA satisfies a log-Sobolev inequality with constant $2/(\alpha^{2}\eta)$ under a mild burn-in condition.
As the log-Sobolev constant quantifies both the tail behaviour and dependence, we can derive a concentration bound for the sample covariance matrix even under dependence as \citet{adamczak2015note}.

We also remark the role of the burn-in condition in Theorem~\ref{thm:concentration}.
First, the condition $m\ge \log(\kappa_{0}/\eta)/\log(1/(1-\alpha\eta))$ is to eliminate the effect of the initial distribution of $X_{0}$ if its log-Sobolev constant is large and thus its tail behaviour is relatively heavy.
Second, the condition $m\ge \left(\log(4(\E[|X_{0}-x^{\ast}|^{2}]+d/\alpha)/\eta n)/\log(1/(1-\alpha\eta))\right)$ is to remove the effect of the non-stationary nature of ULA due to the biased initial distribution.

\section{Sample complexity guarantees for covariance estimation}\label{sec:appl}
In this section, using Theorem \ref{thm:concentration}, we show sample complexity estimates for single-chain ULA and  embarrassingly parallel ULA.
For single-chain ULA, we show sample complexity estimates $n(\delta,\epsilon)$ and $m(\delta,\epsilon)$ (burn-in complexity) to achieve that for any $\epsilon$ and $\delta$, for all $m\ge m(\delta,\epsilon)$ and $n\ge n(\delta,\epsilon)$,
\begin{equation*}
    \Pr\left(\left\|\hat{\Sigma}_{n}-\Cov(\pi)\right\|\le \epsilon\right)\ge 1-\delta.
\end{equation*}
For the embarrassingly parallel ULA, we also derive sample complexity estimates $N_{\text{ep}}(\delta,\epsilon)$ and $m_{\text{ep}}(\delta,\epsilon)$ to let the following bound hold for all $m\ge m_{\text{ep}}(\delta,\epsilon)$ and $N\ge N_{\text{ep}}(\delta,\epsilon)$:
\begin{equation*}
    \Pr\left(\left\|\tilde{\Sigma}_{N}-\Cov(\pi)\right\|\le \epsilon\right)\ge 1-\delta,
\end{equation*}
where $\tilde{\Sigma}_{N}$ is defined as
\begin{equation*}
    \tilde{\Sigma}_{N}=\frac{1}{N}\sum_{k=1}^{N}X_{m+1}^{k}(X_{m+1}^{k})^{\top}-\left(\frac{1}{N}\sum_{k=1}^{N}X_{m+1}^{k}\right)\left(\frac{1}{N}\sum_{k=1}^{N}X_{m+1}^{k}\right)^{\top},
\end{equation*}
and $X_{m+1}^{k}$ is the last iterate of ULA defined in \eqref{eq:def:ep}.
Note that the total sample complexity estimates for single-chain ULA and the embarrassingly parallel ULA are $m(\delta,\epsilon)+n(\delta,\epsilon)$ and $N_{\text{ep}}(\delta,\epsilon)m_{\text{ep}}(\delta,\epsilon)$ respectively.

\subsection{Sample complexity for the unadjusted Langevin algorithm}
We obtain the following high-probability bound for covariance estimation via ULA.
\begin{proposition}\label{prop:comp:ula}
    Suppose that Assumption~\ref{asmp:potential} holds. Fix $\epsilon\le \min\{1,1/\alpha\}$ and assume that $X_{0}=x^{\ast}=\argmin_{x\in\R^{d}}f(x)$.
    For any $\delta>0$, if 
    \begin{equation*}
        \eta\le \frac{\alpha^{3}\epsilon^{2}}{2700\beta^{2}d^{2}},\ n\ge \frac{2^{9}\cdot 3^{2}}{\eta}\left(\frac{9d+4\log(4/\delta)}{\alpha^{3}\epsilon^{2}}\right),\ 
        m\ge \frac{\log((4d)/(\alpha\eta n))}{\log(1/(1-\alpha\eta))}-1,\ 
    \end{equation*}
    then
    \begin{equation*}
        \Pr\left(\left\|\hat{\Sigma}_{n}-\Cov(\pi)\right\|\le \epsilon\right)\ge 1-\delta.
    \end{equation*}
\end{proposition}
The assumption on the initial value $X_{0}=x^{\ast}$ is just for a concise presentation; if one consider other initial values or warm-start, its effect can be eliminated by an additional burn-in phase, whose length is only logarithmic in $|x_0 - x^\ast|$ when $X_0 = x_0$ is deterministic (cold start), or in $\mathbb{E}[|X_0 - x^\ast|^2]$ when $X_0$ is random (warm start).

Using this proposition, we derive the following sample complexity estimate and burn-in complexity estimate: by $\log(1/(1-x))\ge x$ for all $x\in[0,1)$, 
\begin{align*}
    n(\delta,\epsilon)=\calO\left(\frac{\beta^{2}d^{2}(d+\log(1/\delta))}{\alpha^{6}\epsilon^{4}}\right),\ 
    m(\delta,\epsilon)=\calO\left(\frac{\beta^{2}d^{2}}{\alpha^{4}\epsilon^{2}}\right)=\calO(n(\delta,\epsilon)).
\end{align*}
The total sample complexity, the order of $n(\delta,\epsilon)+m(\delta,\epsilon)$, is given as
\begin{equation}\label{eq:comp:ula}
    n(\delta,\epsilon)+m(\delta,\epsilon)=\calO\left(\frac{\beta^{2}d^{2}(d+\log(1/\delta))}{\alpha^{6}\epsilon^{4}}\right).
\end{equation}
To the best of our knowledge, this is the first estimate for the sample complexity of covariance matrix estimation for strongly log-concave smooth distributions via ULA.
This guarantee enjoys the following properties: (i) the estimation error does not deteriorate for large $n$ beyond the sample complexity, which is important in statistics, and (ii) cold start is allowed, which is also important in practice.

\subsection{Comparison with embarrassingly parallel ULA}
As a comparison, we derive the following high-probability guarantee for the embarrassingly parallel ULA.
\begin{proposition}\label{prop:comp:ep}
    Suppose that Assumption~\ref{asmp:potential} holds and fix $\epsilon\le \min\{1,1/\alpha\}$.
    For all $k=1,\ldots,N$, let $X_{0}^{k}=x^{\ast}=\argmin_{x\in\R^{d}}f(x)$.
    For any $\delta>0$, if 
    \begin{equation*}
        \eta\le \frac{\alpha^{3}\epsilon^{2}}{2700\beta^{2}d^{2}},\ m\ge \frac{2\log(48d/(\alpha\epsilon) )}{\log\left(1/(1-\alpha\eta)\right)}-1,\ N\ge \frac{2^{8}\cdot 3^{2}(9d+4\log(4/\delta))}{\alpha^{2}\epsilon^{2}},
    \end{equation*}
    then 
    \begin{equation*}
        \Pr\left(\left\|\tilde{\Sigma}_{N}-\Cov(\pi)\right\|\le \epsilon\right)\ge 1-\delta.
    \end{equation*}
\end{proposition}
We obtain the following sampling complexity estimate: using $\log(1/(1-x))\ge x$ for all $x\in[0,1)$, we obtain that
\begin{equation*}
    m_{\text{ep}}(\delta,\epsilon)=\calO\left(\frac{\beta^{2}d^{2}\log\left(d/(\alpha\epsilon)\right)}{\alpha^{4}\epsilon^{2}}\right),\ N_{\text{ep}}(\delta,\epsilon)=\calO\left(\frac{d+\log \delta^{-1}}{\alpha^{2}\epsilon^{2}}\right),
\end{equation*}
and thus, the total sample complexity is
\begin{equation}\label{eq:comp:ep}
    N_{\text{ep}}(\delta,\epsilon)m_{\text{ep}}(\delta,\epsilon)=\calO\left(\frac{\beta^{2}d^{2}(d+\log \delta^{-1})}{\alpha^{6}\epsilon^{4}}\log\left(\frac{d}{\alpha\epsilon}\right)\right).
\end{equation}

Comparing \eqref{eq:comp:ula} and \eqref{eq:comp:ep}, the difference in the total sample complexity of single-chain ULA and embarrassingly parallel ULA is a logarithmic factor in $d$ and $\epsilon^{-1}$.
This difference comes from the bias term due to non-stationarity.
For single-chain ULA, the bias by non-stationarity is of order $\calO((1/n\eta)(1-\alpha\eta)^{m/2})$; it means that the bias caused by finite burn-in steps $m$ is mitigated by large $n\eta$, which is originally taken to be large to reduce the variance term as seen in Theorem \ref{thm:concentration}.
On the other hand, the corresponding term is of order $\calO((1-\alpha\eta)^{m/2})$ for embarrassingly parallel ULA, which cannot be reduced by increasing $N$.
It means that embarrassingly parallel ULA requires large $m$ for every iteration of the path of ULA.
This difference affects the total sample complexity.

\section*{Acknowledgements}
This work was supported by JSPS KAKENHI Grant Number JP24K02904 and JST PRESTO Grant Number JPMJPR24K7.

\appendix
\section{Proofs}\label{app:proof}
\subsection{Preliminaries}
In this section, we prepare some technical results for the proofs of Propositions \ref{prop:first} and \ref{prop:second}.
The following proposition gives a log-Sobolev constant of the joint distribution of $(X_{m+1},\ldots,X_{m+n})$, which is a technical key tool in the proofs of the propositions.
\begin{proposition}[LSI for joint distributions]\label{prop:lsi}
    Suppose that Assumption~\ref{asmp:potential} holds.
    Moreover, assume the following:
    \begin{itemize}
        \item[(a)] For some $\kappa_{0}\ge0$, the distribution of $X_{0}$ satisfies a log-Sobolev inequality with constant $\kappa_{0}$.
        \item[(b)] $\eta\in(0,1/\beta)$.
        \item[(c)] $m\ge 0\vee(\frac{\log(\kappa_{0}/(2\eta))}{\log(1/(1-\alpha\eta))}-1)$ (with convention $\log0=-\infty$).
    \end{itemize}
    Then, the joint probability distribution of $(X_{m+1},\ldots,X_{m+n})$ satisfies a log-Sobolev inequality with constant $2/(\alpha^{2}\eta)$.
\end{proposition}

Regarding the condition (c), if $\kappa_{0}\le 2\eta$, it is satisfied for any $m\ge0$.
In particular, cold start (i.e., $X_{0}=x_{0}$ for any $x_{0}\in\R^{d}$ with probability 1) satisfies the condition since $\kappa_{0}=0$, and thus we can take $m=0$.

Let us argue the tightness of the estimate $2/(\alpha^{2}\eta)$.
Exemplify a $1$-dimensional stationary Gaussian AR(1) process $Y_{i+1}=(1-\alpha\eta)Y_{i}+\sqrt{2\eta}Z_{i+1}$.
$Y_{i}$ has the autocovariance matrix whose largest eigenvalue approaches to $2/(\alpha^{2}\eta)$ \citep[e.g.,][]{sherman2023eigenstructure}, which is a log-Sobolev constant that matches our estimate.
The log-Sobolev constant of the distribution of $Y_{0}$ is $2/\alpha$ \citep[Theorem 8 of][arXiv v4]{vempala2019rapid}, and the condition (c) is also satisfied for some finite $m\ge0$.
Since the log-Sobolev constant of the joint distribution of $(Y_{1},\ldots,Y_{n})$ is equal to that of $(Y_{m+1},\ldots,Y_{m+n})$ under stationarity, our estimate is tight.

\begin{remark}
    Let us discuss the effect of burn-in period $m$.
    A typical motivation of introducing burn-in period is to reduce the bias by non-stationarity in \eqref{eq:bvd}, which results from the initial distribution largely deviating from the invariant distribution $\pi_{\eta}$.
    Moreover, Proposition \ref{prop:lsi} and its proof indicate that burn-in also helps to mitigate the effect of the initial distribution in terms of the log-Sobolev constant.
\end{remark}

\begin{proof}
Let us define a scaled version of $X_{0}$ and its realization $x_{0}$ as $\tX_{0}=\sqrt{1/\kappa_{0}}X_{0}$ and $\tx_{0}=\sqrt{1/\kappa_{0}}x_{0}$.
Then, the distribution of $\tX_{0}$ satisfies a log-Sobolev inequality with constant $1$.

It is sufficient to show that $(X_{m+1},X_{1},\ldots,X_{m+n})=g(\tX_{0},Z_{1},\ldots,Z_{m+n})$ for a map $g:\R^{d(m+n+1)}\to\R^{dn}$ with $|g(\tx_{0},z_{1},\ldots,z_{n})-g(\tx_{0}',z_{1}',\ldots,z_{n}')|^{2}\le 2/(\alpha^{2}\eta)$ for arbitrary $\tx_{0},\tx_{0}',z_{i},z_{i}'\in\R^{d}$ due to the Gaussian log-Sobolev inequality and tensorization argument \citep{bakry2014analysis}.

We consider a representation $X_{i}$ for $i=m+1,\ldots,m+n$ such as 
\begin{align*}
    X_{i}(\tX_{0},Z_{1},\ldots,Z_{i})&=X_{i}(\tx_{0},z_{1},\ldots,z_{i})|_{\tx_{0}=\tX_{0},z_{j}=Z_{j},j\in[i]}\\
    &=g_{i}(\tx_{0},z_{1},\ldots,z_{m+n})|_{\tx_{0}=\tX_{0},z_{j}=Z_{j},j\in[m+n]}
\end{align*}
and examine its Lipschitz property.
It is obvious that $\nabla_{z_{j}}g_{i}=O$ for any $j> i$, $\nabla_{z_{i}}g_{i}=\sqrt{2\eta}I_{d}$ for any $i=m+1,\ldots,m+n$.
For $j<i$, $\|\nabla_{z_{j}} g_{i}\|\le (1-\alpha\eta)^{i-j}\sqrt{2\eta}$ since for any fixed vectors $z_{j},z_{j}'\in\R^{d}$,
\begin{align*}
    &|g_{i}(\tx_{0},z_{1},\ldots,z_{j-1},z_{j},z_{j+1},\ldots,z_{m+n})-g_{i}(\tx_{0},z_{1},\ldots,z_{j-1},z_{j}',z_{j+1},\ldots,z_{m+n})|\\
    &=|X_{i}(\tx_{0},z_{1},\ldots,z_{j},\ldots,z_{i})-X_{i}(\tx_{0},z_{1},\ldots,z_{j}',\ldots,z_{i})|\\
    &=\left|\left(X_{i-1}(\tx_{0},z_{1},\ldots,z_{j},\ldots,z_{i-1})-\eta\nabla f(X_{i-1}(\tx_{0},z_{1},\ldots,z_{j},\ldots,z_{i-1}))\right)\right.\\
    &\quad\left.-\left(X_{i-1}(\tx_{0},z_{1},\ldots,z_{j}',\ldots,z_{i-1})-\eta\nabla f(X_{i-1}(\tx_{0},z_{1},\ldots,z_{j}',\ldots,z_{i-1}))\right)\right|\\
    &\le (1-\alpha\eta)|X_{i-1}(\tx_{0},z_{1},\ldots,z_{j},\ldots,z_{i-1})-X_{i-1}(\tx_{0},z_{1},\ldots,z_{j}',\ldots,z_{i-1})|\\
    &\le (1-\alpha\eta)^{i-j}|X_{j}(\tx_{0},z_{1},\ldots,z_{j-1})-X_{j}(\tx_{0},z_{1},\ldots,z_{j-1}')|\\
    &= \sqrt{2\eta}(1-\alpha\eta)^{i-j}|z_{j}-z_{j}'|,
\end{align*}
where we used the fact that if $\eta\le 1/\beta$ and $f$ is $\alpha$-convex, for any $x,x'\in\R^{d}$, $|x-\eta\nabla f(x)-x'+\eta\nabla f(x')|\le (1-\alpha\eta)|x-x'|$.
Similarly, we have that for any $\tx_{0},\tx_{0}'$,
\begin{equation*}
    |g_{i}(\tx_{0},z_{1},\ldots,z_{m+n})-g_{i}(\tx_{0}',z_{1},\ldots,z_{m+n})|\le \sqrt{\kappa_{0}}(1-\alpha\eta)^{i}|\tx_{0}-\tx_{0}'|.
\end{equation*}

We now evaluate the Lipschitz constant of $g$.
For arbitrary $\tx_{0},\tx_{0}',z_{j},z_{j}'\in\R^{d}$,
\begin{align*}
    &|g(\tx_{0},z_{1},\ldots,z_{m+n})-g(\tx_{0}',z_{1}',\ldots,z_{m+n}')|^{2}\\
    &=\sum_{i=m+1}^{m+n}|g_{i}(\tx_{0},z_{1},\ldots,z_{m+n})-g_{i}(\tx_{0}',z_{1}',\ldots,z_{m+n}')|^{2}\\
    &\le \sum_{i=m+1}^{m+n}\left((1-\alpha\eta)^{i}\sqrt{\kappa_{0}}|\tx_{0}-\tx_{0}'|+\sum_{j=1}^{i}(1-\alpha\eta)^{i-j}\sqrt{2\eta}|z_{j}-z_{j}'|\right)^{2}
    \\
    &\le \sum_{i=m+1}^{m+n}\left(\sum_{j=0}^{i}(1-\alpha\eta)^{i-j}\right)\left(\kappa_{0}(1-\alpha\eta)^{i}|\tx_{0}-\tx_{0}'|^{2}+(2\eta)\sum_{j=1}^{i}(1-\alpha\eta)^{i-j}|z_{j}-z_{j}'|^{2}\right)\\
    &\le \frac{1}{\alpha\eta}\sum_{i=m+1}^{m+n}\left(\kappa_{0}(1-\alpha\eta)^{i}|\tx_{0}-\tx_{0}'|^{2}+(2\eta)\sum_{j=1}^{i}(1-\alpha\eta)^{i-j}|z_{j}-z_{j}'|^{2}\right)\\
    &\le \frac{\kappa_{0}}{\alpha\eta}\sum_{i=m+1}^{m+n}(1-\alpha\eta)^{i}|\tx_{0}-\tx_{0}'|^{2}+\frac{2}{\alpha}\sum_{i=m+1}^{m+n}\sum_{j=1}^{i}(1-\alpha\eta)^{i-j}|z_{j}-z_{j}'|^{2}\\
    &\le \frac{\kappa_{0}}{\alpha^{2}\eta^{2}}(1-\alpha\eta)^{m+1}|\tx_{0}-\tx_{0}'|^{2}+\frac{2}{\alpha}\sum_{j=1}^{m+n}\sum_{i=j\vee(m+1)}^{m+n}(1-\alpha\eta)^{i-j}|z_{j}-z_{j}'|^{2}\\
    &\le \frac{2}{\alpha^{2}\eta}\left(|\tx_{0}-\tx_{0}'|^{2}+\sum_{j=1}^{m+n}|z_{j}-z_{j}'|^{2}\right),
\end{align*}
where we used $\eta\le 1/\beta\le 1/\alpha$ and the fact $(\kappa_{0}/(\alpha^{2}\eta^{2}))(1-\alpha\eta)^{m+1}\le 2/(\alpha^{2}\eta)$ since if $\kappa_{0}>2\eta$,
\begin{equation*}
    (1-\alpha\eta)^{m+1}\le \frac{2\eta}{\kappa_{0}}\iff m\ge \frac{\log(\kappa_{0}/(2\eta))}{\log(1/(1-\alpha\eta))}-1,
\end{equation*}
and if $\kappa_{0}\le 2\eta$, $(\kappa_{0}/(\alpha^{2}\eta^{2}))(1-\alpha\eta)^{m+1}\le 2/(\alpha^{2}\eta)$ for any $m\ge0$.
This complete the proof.
\end{proof}

\begin{lemma}[Subexponentiality of quadratic forms; \citealp{nakakita2025corrigendum}]\label{lem:nai25}
    Suppose that $A$ is a $dn\times dn $ positive semi-definite matrix and $Y$ is a $dn$-dimensional random vector whose distribution satisfies a log-Sobolev inequality with constant $\kappa$. Then, for any $t\in(-\infty,1/(2\|A\|\kappa))$,
    \begin{equation*}
        \E\left[\exp\left(tY^{\top}AY\right)\right]\le \exp\left(\left(t+\frac{2\|A\|\kappa t^{2}}{1-2\|A\|\kappa t}\right)\E[Y^{\top}AY]\right).
    \end{equation*}
\end{lemma}

The next lemma is inspired by Theorem 8 of \citet[arXiv v4]{vempala2019rapid}.
\begin{lemma}[LSI for marginal distributions]\label{lem:vem19}
    Suppose that Assumption~\ref{asmp:potential} holds and the distribution of $X_{0}$ satisfies a log-Sobolev inequality with constant $\kappa_{0}$. For any $\eta\in(0,1/\beta)$ and $i\ge 0$, the distribution of $X_{i}$ satisfies a log-Sobolev inequality with constant $(1-\alpha\eta)^{2i}\kappa_{0}+(1-(1-\alpha\eta)^{2i})(2/\alpha)$.
    In particular, if $i\ge \frac{\log(\alpha\kappa_{0})}{2\log(1/(1-\alpha\eta))}$ (with convention $\log0=-\infty$), then the distribution of $X_{i}$ satisfies a log-Sobolev inequality with constant $3/\alpha$.
\end{lemma}

\begin{proof}
    Since $\eta\le 1/\beta$, $x-\eta\nabla f(x)$ is $(1-\alpha\eta)$-Lipschitz.
    Let $\alpha$ denote a log-Sobolev constant of the distribution of $X_{i}$.
    Also note the following fact \citep{bakry2014analysis}: if two random variables $U$ and $V$ are independent and their distributions satisfy log-Sobolev inequalities with constants $\kappa_{U}$ and $\kappa_{V}$, respectively, then the distribution of $U+V$ satisfies a log-Sobolev inequality with constant $\kappa_{U}+\kappa_{V}$.
    By the Gaussian log-Sobolev inequality and the Lipschitz property, for any $i\ge0$,
    \begin{equation*}
        \alpha_{i+1}\le (1-\alpha\eta)^{2}\alpha_{i}+2\eta\le (1-\alpha\eta)^{2(i+1)}\kappa_{0}+(1-(1-\alpha\eta)^{2(i+1)})(2/\alpha),
    \end{equation*}
    which is the desired result.
\end{proof}

Let us recall the definition of a Poincar\'{e} inequality. 
A probability distribution $\nu$ on $\R^{d}$ is said to satisfy a Poincar\'{e} inequality with constant $C\ge 0$ (denoted by PI($C$)) if for any smooth function $g:\R^{d}\to\R$ with $\int_{\R^{d}}|g(x)|^{2}\mu(\diff x)<\infty$,
\begin{equation*}
    \int_{\R^{d}}g^{2}(x)\nu(\diff x)-\left(\int_{\R^{d}}g(x)\nu(\diff x)\right)^{2}\le C\int_{\R^{d}}|\nabla g(x)|^{2}\nu(\diff x).
\end{equation*}
Note that if a distribution satisfies a log-Sobolev inequality with constant $C$, it also satisfies a Poincar\'{e} inequality with constant $C$ \citep{bakry2014analysis}.
\begin{lemma}[Bound on the spectral norm of a covariance]\label{lem:pi}
    For any $d$-dimensional random variable $\{Y_{i}\}_{i=1}^{n}$ such that that the joint probability distribution of $(Y_{1},\ldots,Y_{d})$ on $\R^{dn}$ satisfies a Poincar\'{e} inequality with constant $C>0$, it holds that
    \begin{equation*}
        \left\|\E\left[\left(\frac{1}{n}\sum_{i=1}^{n}Y_{i}\right)\left(\frac{1}{n}\sum_{i=1}^{n}Y_{i}\right)^{\top}-\E\left[\frac{1}{n}\sum_{i=1}^{n}Y_{i}\right]\E\left[\frac{1}{n}\sum_{i=1}^{n}Y_{i}\right]^{T}\right]\right\|\le \frac{C}{n}.
    \end{equation*}
\end{lemma}

\begin{proof}
    It sufficient to show that for any $u\in\bbS^{d-1}$, $g_{u}:\R^{dn}\to\R$ such that $g_{u}(y_{1},\ldots,y_{n})=(\frac{1}{n}\sum_{i=1}^{n}u^{\top}y_{i})$ for $y_{i}\in\R^{d}$ satisfies $|\nabla g_{u}|\le 1/\sqrt{n}$.
    Indeed, for any $j\in[n]$,\begin{equation*}
        \nabla_{y_{j}}g_{u}(y_{1},\ldots,y_{n})=\frac{1}{n}u,
    \end{equation*}
    and thus $|\nabla g_{u}(y_{1},\ldots,y_{n})|^{2}=\sum_{j=1}^{n}|\nabla_{y_{j}}g_{u}(y_{1},\ldots,y_{n})|^{2}=1/n$.
    This complete the proof.
\end{proof}

Before the proof of Proposition \ref{prop:comp:ula}, let us prepare notation.
Let the Markov kernel $R_{\eta}$ of ULA with step size $\eta$ such that for all $A\in\calB(\R^{d})$ and $x\in\R^{d}$ by 
\begin{equation}\label{eq:def:kernel}
    R_{\eta}(x,A)=\int_{A}\frac{1}{(4\pi\eta)^{d/2}}\exp\left(-\frac{1}{4\eta}\left|y-x+\eta\nabla f(x)\right|^{2}\right)\diff y.
\end{equation}
$\delta_{x}R_{\eta}^{i}$ denotes the distribution of $X_{i}$ given $X_{0}=x$ as usual.

\begin{proposition}[Bounds on moments, \citealp{durmus2019high}]\label{prop:moment}
    Suppose that Assumption~\ref{asmp:potential} holds. 
    The following holds:
    \begin{enumerate}
        \item[(a)] $\int_{\R^{d}}|x-x^{\ast}|^{2}\pi(\diff x)\le d/\alpha$.
        \item[(b)] For any $\eta \in (0, 1/\beta)$, $R_{\eta}$ has a unique stationary distribution $\pi_{\eta}$, and it holds that $\int_{\R^{d}}|x - x^{\ast}|^{2} \pi_{\eta}(\diff x) \le 2d/\alpha$.
        \item[(c)] For any $\eta \in (0, 1/\beta)$ and $i\in\N$, $\int_{\R^{d}}|y - x^{\ast}|^{2} \delta_{x}R_{\eta}^{i}(\diff y) \le (1-\alpha\eta)^{i}|x-x^{\ast}|^{2}+2(d/\alpha)(1-(1-\alpha\eta)^{i})$.
    \end{enumerate}
    
\end{proposition}

\begin{proposition}[Bounds on 2-Wasserstein distances, \citealp{durmus2019high}]\label{prop:wasserstein}
    Suppose that Assumption~\ref{asmp:potential} holds and $\eta\in(0,1/\beta)$.
    \begin{enumerate}
        \item[(a)] $W_{2}^{2}(\pi_{\eta},\pi)\le 4(\beta/\alpha)^{2}\eta(2d + d^{2}\beta^{2}\eta/\alpha + d^{2}\eta^{2}/6)$.
        \item[(b)] For any $i\in\N$, $W_{2}^{2}(\delta_{x}R_{\eta}^{i},\pi_{\eta})\le (1-\alpha\eta)^{i}(|x-x^{\ast}|^{2}+2d/\alpha)$.
    \end{enumerate}
\end{proposition}

\begin{lemma}[Bounds on the difference of matrices]\label{lem:cov}
    For any probability distributions $\mu$ and $\nu$ on $\R^{d}$ and any vector $z\in\R^{d}$,
    \begin{align*}
        \textnormal{(a) }&\|\E_{X\sim\mu}[(X-z)(X-z)^{\top}]-\E_{Y\sim\nu}[(Y-z)(Y-z)^{\top}]\|\\
        &\quad\le \left(\sqrt{\int|x-z|^{2}\mu(\diff x)}+\sqrt{\int|y-z|^{2}\nu(\diff y)}\right)W_{2}(\mu,\nu),\\
        \textnormal{(b) }&\|\Cov(\mu)-\Cov(\nu)\|\le 2\left(\sqrt{\int|x-z|^{2}\mu(\diff x)}+\sqrt{\int|y-z|^{2}\nu(\diff y)}\right)W_{2}(\mu,\nu).
    \end{align*}
\end{lemma}

\begin{proof}
    We only show (b) since the proof of (b) contains that of (a).
    Let $X\sim\mu$ and $Y\sim\nu$ be random variables defined on a common probability space such that $W_{2}(\mu,\nu)=\E[|X-Y|^{2}]^{1/2}$.
    The Cauchy--Schwarz inequality and triangular inequality yield
    \begin{align*}
        \|\Cov(\mu)-\Cov(\nu)\|&=\left\|\E[XX^{\top}]-\E[YY^{\top}]-\E[X]\E[X]^{\top}+\E[Y]\E[Y]^{\top}\right\|\\
        &\le \left\|\E[(X-z)(X-z)^{\top}]-\E[(Y-z)(Y-z)^{\top}]\right\|\\
        &\quad+\left\|\E[X-z]\E[X-z]^{\top}-\E[Y-z]\E[Y-z]^{\top}\right\|\\
        &\le \sup_{u\in\bbS^{d-1}}\left|\E[(u^{\top}(X-z))^{2}]-\E[(u^{\top}(Y-z))^{2}]\right|\\
        &\quad+\left|\E[X-z]+\E[Y-z]\right|\left|\E[X-z]-\E[Y-z]\right|\\
        &\le \sup_{u\in\bbS^{d-1}}\left|\E[(u^{\top}((X-z)+(Y-z)))(u^{\top}(X-Y))]\right|\\
        &\quad+\left(\E[|X-z|]+\E[|Y-z|]\right)\E[\left|X-Y\right|]\\
        &\le \E[|(X-z)+(Y-z)|^{2}]^{1/2}\E[|X-Y|^{2}]^{1/2}\\
        &\quad+\left(\E[|X-z|^{2}]^{1/2}+\E[|Y-z|^{2}]^{1/2}\right)\E[\left|X-Y\right|^{2}]^{1/2}\\
        &\le 2\left(\E[|X-z|^{2}]^{1/2}+\E[|Y-z|^{2}]^{1/2}\right)W_{2}(\mu,\nu).
    \end{align*}
    This gives the desired result.
\end{proof}

As we frequently use the following elementary facts, we state them as lemmas.
\begin{lemma}\label{lem:elem1}
    For any $a,b,c>0$, it holds that
    \begin{equation*}
        \min_{\lambda\in[0,c]}\left(a\lambda^{2}-b\lambda\right)\le -\left(\frac{b^{2}}{4a}\wedge \frac{bc}{2}\right).
    \end{equation*}
\end{lemma}
\begin{proof}
    Let $f(\lambda)=a\lambda^{2}-b\lambda$.
    If $b/(2a)\le c$, then the minimum is attained at $\lambda=b/(2a)$, and thus the minimum is $-b^{2}/(4a)$, and it clearly holds $b^{2}/(4a)\le bc/2$.
    If $b/(2a)>c$, then the minimum is attained at $\lambda=c$, and thus the minimum is $ac^{2}-bc\le -(bc/2)$, and it also holds $b^{2}/(4a)\ge bc/2$.
    This complete the proof.
\end{proof}

\begin{lemma}\label{lem:elem2}
    For any $a,b,c>0$ and $\delta\in(0,1],t\ge 0$, it holds that
    \begin{equation*}
        \exp\left(c-(at^{2})\wedge (bt)\right)\le \delta \iff t\ge \frac{c+\log\left(1/\delta\right)}{b}\vee \sqrt{\frac{c+\log\left(1/\delta\right)}{a}}.
    \end{equation*}
\end{lemma}

\begin{proof}
    We have
    \begin{align*}
        &\exp\left(c-(at^{2})\wedge (bt)\right)\le \delta 
        \iff (at^{2})\wedge (bt)\ge c+\log\left(1/\delta\right)\\
        &\iff at^{2}\ge c+\log\left(1/\delta\right)\text{ and }bt\ge c+\log\left(1/\delta\right)
        \iff t\ge \frac{c+\log\left(1/\delta\right)}{b}\vee \sqrt{\frac{c+\log\left(1/\delta\right)}{a}}.
    \end{align*}
    This complete the proof.
\end{proof}

\subsection{Proof of Theorem \ref{thm:concentration}}
Before the proof of Theorem \ref{thm:concentration}, we define translated versions of $X_{i}$ as $X_{i}^{\ctr}=X_{i}-\frac{1}{n}\sum_{i=m+1}^{m+n}\E[X_{i}]$ for all $i=m+1,\ldots,m+n$ and the corresponding sample mean as:
\begin{equation*}
    \bar{X}_{n}^{\ctr}:=\frac{1}{n}\sum_{i=m+1}^{m+n}X_{i}^{\ctr}=\frac{1}{n}\sum_{i=m+1}^{m+n}\left(X_{i}-\E\left[\frac{1}{n}\sum_{i=m+1}^{m+n}X_{i}\right]\right).
\end{equation*}
Note that the sample covariance matrix $\hat{\Sigma}_{n}$ can be rewritten as
\begin{equation*}
    \hat{\Sigma}_{n}=\frac{1}{n}\sum_{i=m+1}^{m+n}X_{i}^{\ctr}(X_{i}^{\ctr})^{\top}-\bar{X}_{n}^{\ctr}(\bar{X}_{n}^{\ctr})^{\top}.
\end{equation*}

Theorem \ref{thm:concentration} is a direct consequence of the following concentration inequalities for the squared sample means and sample second moment matrices of ULA.

We first present a concentration bound for squared sample means.
\begin{proposition}\label{prop:first}
    Suppose that Assumption~\ref{asmp:potential} holds, $\eta\in(0,1/\beta)$, and the distribution of $X_{0}$ satisfies a log-Sobolev inequality with constant $\kappa_{0}$.
    If $m\ge 0\vee(\frac{\log(\kappa_{0}/(2\eta))}{\log(1/(1-\alpha\eta))}-1)$ (with convention $\log0=-\infty$), then for any $\delta>0$, with probability at least $1-2\delta$,
    \begin{equation*}
        \left\|\bar{X}_{n}^{\ctr}(\bar{X}_{n}^{\ctr})^{\top}-\E\left[\bar{X}_{n}^{\ctr}(\bar{X}_{n}^{\ctr})^{\top}\right]\right\|\le \frac{8 (9 d+4\log\delta^{-1})}{\alpha^{2}\eta n}.
    \end{equation*}
\end{proposition}
\begin{proof}
    For any $u\in\bbS^{d-1}$ and $\lambda\le \alpha^{2}\eta n/8$, using Proposition \ref{prop:lsi} and Lemma \ref{lem:nai25} with $A$ such that 
    \begin{equation*}
        A = \frac{1}{n^{2}}\left[\begin{matrix}
            uu^{\top} & uu^{\top} & \cdots & uu^{\top}\\
            uu^{\top} & uu^{\top} & \cdots & uu^{\top}\\
            \vdots & \vdots & \ddots & \vdots\\
            uu^{\top} & uu^{\top} & \cdots & uu^{\top}
        \end{matrix}\right]
    \end{equation*}
    (note that the translation by $(1/n)\sum_{i=m+1}^{m+n}\E[X_{i}]$ does not affect the estimate for the log-Soblev constant and $\|A\|=1/n$), we have
    \begin{equation*}
        \E\left[\exp\left(\lambda\left(\frac{1}{n}\sum_{i=m+1}^{m+n}u^{\top}X_{i}^{\ctr}\right)^{2}-\E\left[\left(\frac{1}{n}\sum_{i=m+1}^{m+n}u^{\top}X_{i}^{\ctr}\right)^{2}\right]\right)\right]\le \exp\left(\frac{8\lambda^{2}}{\alpha^{2}\eta n}\cdot\frac{2}{\alpha^{2}\eta n}\right),
    \end{equation*}
    where it follows from Proposition \ref{prop:lsi}, Lemma \ref{lem:pi}, and the resulting bound such that
    \begin{equation*}
        \E\left[\left(\frac{1}{n}\sum_{i=m+1}^{m+n}u^{\top}X_{i}^{\ctr}\right)^{2}\right]\le \frac{2}{\alpha^{2}\eta n}.
    \end{equation*}
    Then, using the net argument \citep{vershynin2018high}, for some $1/4$-net $\calN\subset\bbS^{d-1}$ with $|\calN|\le 9^{d}$, for any $\lambda\in[0,\alpha^{2}\eta n/8]$ and $t\ge0$,
    \begin{align*}
        &\Pr\left(\left\|\bar{X}_{n}^{\ctr}(\bar{X}_{n}^{\ctr})^{\top}-\E\left[\bar{X}_{n}^{\ctr}(\bar{X}_{n}^{\ctr})^{\top}\right]\right\|\ge t\right)\\
        &\le \sum_{u\in\calN}\Pr\left(\left|\frac{1}{n}\sum_{i=m+1}^{m+n}(u^{\top}X_{i}^{\ctr})^{2}-\E[(u^{\top}X_{i}^{\ctr})^{2}]\right|\ge t/2\right)\\
        &\le 2\cdot 9^{d}\exp\left(\frac{16\lambda^{2}}{\alpha^{4}\eta^{2}n^{2}}-\frac{\lambda t}{2}\right).
    \end{align*}
    Lemma \ref{lem:elem1} (with $a=16/(\alpha^{4}\eta^{2}n^{2})$, $b=t/2$, $c=\alpha^{2}\eta n/8$) yields that
    \begin{equation*}
        \Pr\left(\left\|\bar{X}_{n}^{\ctr}(\bar{X}_{n}^{\ctr})^{\top}-\E\left[\bar{X}_{n}^{\ctr}(\bar{X}_{n}^{\ctr})^{\top}\right]\right\|\ge t\right)\le 2\exp\left(d\log 9 -  \left(\frac{\alpha^{4}\eta^{2} n^{2}}{256} t^{2}\right)\wedge\left(\frac{\alpha^{2}\eta n}{32}t\right)\right).
    \end{equation*}
    Hence, using Lemma \ref{lem:elem2} and the fact that and $\sqrt{256(d\log 9+\log\delta^{-1})}\le 32(d\log 9+\log\delta^{-1})$ with $\delta\in(0,1/2]$, we derive that for any $\delta>0$,
    \begin{align*}
        \Pr\left(\left\|\bar{X}_{n}^{\ctr}(\bar{X}_{n}^{\ctr})^{\top}-\E\left[\bar{X}_{n}^{\ctr}(\bar{X}_{n}^{\ctr})^{\top}\right]\right\|\ge \frac{32 (d\log9+\log\delta^{-1})}{\alpha^{2}\eta n}\right)
        \le 2\delta.
    \end{align*}
    This complete the proof with $4\log9\le 9$.
\end{proof}

We also give a concentration bound for the sample second moment matrix around its expectation.
\begin{proposition}\label{prop:second}
    Suppose that Assumption~\ref{asmp:potential} holds, $\eta\in(0,1/\beta)$, and the distribution of $X_{0}$ satisfies a log-Sobolev inequality with constant $\kappa_{0}$.
    If $m$ satisfies
    \begin{equation*}
        m\ge 0\vee\left(\frac{\log(\kappa_{0}/\eta)}{\log(1/(1-\alpha\eta))}-1\right)\vee \left(\frac{\log(4(\E[|X_{0}-x^{\ast}|^{2}]+d/\alpha)/\eta n)}{\log(1/(1-\alpha\eta))}-1\right)
    \end{equation*}
    (with convention $\log0=-\infty$), then for any $\delta>0$, with probability at least $1-2\delta$,
    \begin{align*}
        &\left\|\frac{1}{n}\sum_{i=m+1}^{m+n}X_{i}^{\ctr}(X_{i}^{\ctr})^{\top}-\E\left[\frac{1}{n}\sum_{i=m+1}^{m+n}X_{i}^{\ctr}(X_{i}^{\ctr})^{\top}\right]\right\|\\
        &\le \frac{1}{\alpha}\left(\sqrt{128\cdot \frac{9d+4\log\delta^{-1}}{\alpha\eta n}}\vee 8\left(\frac{9 d+4\log\delta^{-1}}{\alpha\eta n}\right)\right).
    \end{align*}
\end{proposition}

\begin{proof}
    Note that $m$ satisfies
    \begin{align*}
        m&\ge \left(\frac{\log(\kappa_{0}/(2\eta))}{\log(1/(1-\alpha\eta))}-1\right)\vee \left(\frac{\log(4(\E[|X_{0}-x^{\ast}|^{2}]+d/\alpha)/\eta n)}{\log(1/(1-\alpha\eta))}-1\right)\\
        &\quad\vee \left(\frac{\log(\alpha \kappa_{0})}{2\log(1/(1-\alpha\eta))}-1\right).
    \end{align*}
    Hence, Lemma \ref{lem:vem19}, Proposition \ref{prop:wasserstein}, and the assumption on $m$ yield that
    \begin{align*}
        &\E\left[\frac{1}{n}\sum_{i=m+1}^{m+n}(u^{\top}X_{i}^{\ctr})^{2}\right]\\
        &=\frac{1}{n}\sum_{i=m+1}^{m+n}\E\left[(u^{\top}(X_{i}-\E[X_{i}])^{2}\right]
        +\frac{1}{n}\sum_{i=m+1}^{m+n}\left(u^{\top}\left(\E[X_{i}]-\frac{1}{n}\sum_{i=m+1}^{m+n}\E[X_{i}]\right)\right)^{2}\\
        &\le \frac{3}{\alpha}+\frac{2}{n}\sum_{i=m+1}^{m+n}\left(u^{\top}\left(\E[X_{i}]-\E_{\pi_{\eta}}\left[X\right]\right)\right)^{2}+2\left(u^{\top}\left(\frac{1}{n}\sum_{i=m+1}^{m+n}\E[X_{i}]-\E_{\pi_{\eta}}\left[X\right]\right)\right)^{2}\\
        &\le \frac{3}{\alpha}+\frac{4}{n}\sum_{i=m+1}^{m+n}\left(u^{\top}\left(\E[X_{i}]-\E_{\pi_{\eta}}\left[X\right]\right)\right)^{2}
        \le \frac{3}{\alpha}+\frac{4}{n}\sum_{i=m+1}^{m+n}\left\|\E[X_{i}]-\E_{\pi_{\eta}}\left[X\right]\right\|^{2}\\
        &\le \frac{3}{\alpha}+\frac{4}{n}\sum_{i=m+1}^{m+n}W_{2}^{2}(X_{i},\pi_{\eta})\le \frac{3}{\alpha}+\frac{4}{n}\sum_{i=m+1}^{m+n}(1-\alpha\eta)^{i}\left(\E[|X_{0}-x^{\ast}|^{2}]+\frac{d}{\alpha}\right)\\
        &\le \frac{3}{\alpha}+\frac{4}{\alpha \eta n}(1-\alpha\eta)^{m+1}\left(\E[|X_{0}-x^{\ast}|^{2}]+\frac{d}{\alpha}\right)\le \frac{4}{\alpha}.
    \end{align*}
    Using Proposition \ref{prop:lsi} and Lemma \ref{lem:nai25}, for any $u\in\bbS^{d-1}$ and $\lambda\le \alpha^{2}\eta n/8$, we have
    \begin{equation*}
        \E\left[\exp\left(\frac{\lambda}{n}\sum_{i=m+1}^{m+n}\left((u^{\top}X_{i}^{\ctr})^{2}-\E[(u^{\top}X_{i}^{\ctr})^{2}]\right)\right)\right]\le \exp\left(\frac{8\lambda^{2}}{\alpha^{2}\eta n}\cdot\frac{4}{\alpha}\right).
    \end{equation*}
    Then, \citet{vershynin2018high} gives that for some $1/4$-net $\calN\subset\bbS^{d-1}$ with $|\calN|\le 9^{d}$, for any $\lambda\in [0,\alpha^{2}\eta n/8]$,
    \begin{align*}
        &\Pr\left(\left\|\frac{1}{n}\sum_{i=m+1}^{m+n}X_{i}^{\ctr}(X_{i}^{\ctr})^{\top}-\E\left[\frac{1}{n}\sum_{i=m+1}^{m+n}X_{i}^{\ctr}(X_{i}^{\ctr})^{\top}\right]\right\|\ge t\right)\\
        &\le \sum_{u\in\calN}\Pr\left(\left|\frac{1}{n}\sum_{i=m+1}^{m+n}(u^{\top}X_{i}^{\ctr})^{2}-\E\left[\frac{1}{n}\sum_{i=m+1}^{m+n}(u^{\top}X_{i}^{\ctr})^{2}\right]\right|\ge \frac{t}{2}\right)\\
        &\le 2\cdot 9^{d}\exp\left(\frac{32\lambda^{2}}{\alpha^{3}\eta n}-\frac{\lambda t}{2}\right).
    \end{align*}
    Lemma \ref{lem:elem1} (with $a=32/(\alpha^{3}\eta n)$, $b=t/2$, $c=\alpha^{2}\eta n/8$) yields that
    \begin{align*}
        &\Pr\left(\left\|\frac{1}{n}\sum_{i=m+1}^{m+n}X_{i}^{\ctr}(X_{i}^{\ctr})^{\top}-\E\left[\frac{1}{n}\sum_{i=m+1}^{m+n}X_{i}^{\ctr}(X_{i}^{\ctr})^{\top}\right]\right\|\ge t\right)\\
        &\le 2\exp\left(d\log 9-\left(\frac{\alpha^{3}\eta n}{512}t^{2}\right)\wedge \left(\frac{\alpha^{2}\eta n}{32}t\right)\right).
    \end{align*}
    Hence, by Lemma \ref{lem:elem2}, for any $\delta>0$, with probability at least $1-2\delta$,
    \begin{align*}
        &\left\|\frac{1}{n}\sum_{i=m+1}^{m+n}X_{i}^{\ctr}(X_{i}^{\ctr})^{\top}-\E\left[\frac{1}{n}\sum_{i=m+1}^{m+n}X_{i}^{\ctr}(X_{i}^{\ctr})^{\top}\right]\right\|&\\
        &\le \frac{1}{\alpha}\left(\sqrt{\frac{512(d\log9+\log\delta^{-1})}{\alpha\eta n}}\vee \left(\frac{32(d\log9+\log\delta^{-1})}{\alpha\eta n}\right)\right).
    \end{align*}
    We derive the proof using $4\log9\le 9$.
\end{proof}

\begin{proof}[Proof of Theorem \ref{thm:concentration}]
    Consider on the event that both the inequalities in Propositions \ref{prop:first} and \ref{prop:second} hold; this event happens with probability at least $1-4\delta$. 
    Then, we have
    \begin{align*}
        \left\|\hat{\Sigma}_{n}-\E[\hat{\Sigma}_{n}]\right\|
        &\le \left\|\frac{1}{n}\sum_{i=m+1}^{m+n}X_{i}^{\ctr}(X_{i}^{\ctr})^{\top}-\E\left[\frac{1}{n}\sum_{i=m+1}^{m+n}X_{i}^{\ctr}(X_{i}^{\ctr})^{\top}\right]\right\|\\
        &\quad+\left\|\bar{X}_{n}^{\ctr}(\bar{X}_{n}^{\ctr})^{\top}-\E\left[\bar{X}_{n}^{\ctr}(\bar{X}_{n}^{\ctr})^{\top}\right]\right\|\\
        &\le \frac{8}{\alpha}\left(\sqrt{\frac{2(9d+4\log\delta^{-1})}{\alpha\eta n}}\vee \left(\frac{9 d+4\log\delta^{-1}}{\alpha\eta n}\right)\right)+\frac{8 (9 d+4\log\delta^{-1})}{\alpha^{2}\eta n}\\
        &\le \frac{16}{\alpha}\left(\sqrt{\frac{2(9d+4\log\delta^{-1})}{\alpha\eta n}}\vee \left(\frac{9 d+4\log\delta^{-1}}{\alpha\eta n}\right)\right).
    \end{align*}
    This is the desired conclusion.
\end{proof}

\subsection{Proof of Proposition \ref{prop:comp:ula}}
We now give the proof of Proposition \ref{prop:comp:ula}.

\begin{proof}[Proof of Proposition \ref{prop:comp:ula}]
    We employ the bias--variance decomposition \eqref{eq:bvd}.

    (Step 1) Since $\epsilon\le \min\{1,1/\alpha\}$ and $\eta\le \epsilon^{2}\alpha^{3}/(2700\beta^{2}d^{2})$, by Proposition \ref{prop:moment}, Proposition \ref{prop:wasserstein}, and Lemma \ref{lem:cov}-(b), we have
    \begin{align*}
        \|\Cov(\pi_{\eta})-\Cov(\pi)\|&\le 2\left(\sqrt{\int|x-x^{\ast}|^{2}\pi_{\eta}(\diff x)}+\sqrt{\int|y-x^{\ast}|^{2}\pi(\diff y)}\right)W_{2}(\pi_{\eta},\pi)\\
        &\le 2\left(\sqrt{2d/\alpha}+\sqrt{d/\alpha}\right)\sqrt{4(\beta/\alpha)^{2}\eta(2d + d^{2}\beta^{2}\eta/\alpha + d^{2}\eta^{2}/6)}\\
        &\le \sqrt{100(d/\alpha)(\beta/\alpha)^{2}\eta(2d + d^{2}\beta^{2}\eta/\alpha + d^{2}\eta^{2}/6)}\\
        &\le \sqrt{\left(\frac{100\beta^{2}d}{\alpha^{3}}\right)\eta\left(2d +\frac{\alpha^{3}\beta^{2}\epsilon^{2}d^{2}}{2700\alpha\beta^{2}d^{2}} + \frac{\epsilon^{4}\alpha^{6}d^{2}}{(6\cdot 2700^{2})\beta^{4}d^{4}}\right)}\\
        &\le \sqrt{\left(\frac{100\beta^{2}d}{\alpha^{3}}\right)\eta\left(2d + \frac{1}{2700} + \frac{1}{6\cdot 2700^{2}}\right)}\\
        &\le \sqrt{\left(\frac{300\beta^{2}d^{2}}{\alpha^{3}}\right)\eta}\le \frac{\epsilon}{3}.
    \end{align*}

    (Step 2)
    Lemma \ref{lem:pi} (with the fact that a distribution satisfying an LSI($C$) for some $C\ge0$ also satisfies a PI($C$)), Propositions \ref{prop:moment} and \ref{prop:wasserstein}, and Lemma \ref{lem:cov} lead to
    \begin{align*}
    \|\E[\hat{\Sigma}_{n}]-\Cov(\pi_{\eta})\|
    &\le \left\|\frac{1}{n}\sum_{i=m+1}^{m+n}\E[(X_{i}-x^{\ast})(X_{i}-x^{\ast})^{\top}]-\E_{\pi_{\eta}}[(X-x^{\ast})(X-x^{\ast})^{\top}]\right\|\\
    &\quad+\left\|\E[(\bar{X}_{n}-x^{\ast})(\bar{X}_{n}-x^{\ast})^{\top}]-\E_{\pi_{\eta}}[X-x^{\ast}]\E_{\pi_{\eta}}[X-x^{\ast}]^{\top}\right\|\\ 
    &\le \frac{1}{n}\sum_{i=m+1}^{m+n}\left\|\E[(X_{i}-x^{\ast})(X_{i}-x^{\ast})^{\top}]-\E_{\pi_{\eta}}[(X-x^{\ast})(X-x^{\ast})^{\top}]\right\|\\
    &\quad+\left\|\E[(\bar{X}_{n}-x^{\ast})(\bar{X}_{n}-x^{\ast})^{\top}]-\E[\bar{X}_{n}-x^{\ast}]\E[\bar{X}_{n}-x^{\ast}]^{\top}\right\|\\
    &\quad+\left\|\E[\bar{X}_{n}-x^{\ast}]\E[\bar{X}_{n}-x^{\ast}]^{\top}-\E_{\pi_{\eta}}[X-x^{\ast}]\E_{\pi_{\eta}}[X-x^{\ast}]^{\top}\right\|\\ 
    &\le \frac{1}{n}\sum_{i=m+1}^{m+n}W_{2}(\delta_{x^{\ast}}R_{\eta}^{i},\pi_{\eta})\left(\sqrt{\E[|X_{i}-x^{\ast}|^{2}]}+\sqrt{\E_{\pi_{\eta}}[|X-x^{\ast}|^{2}]}\right)\\
    &\quad+\frac{2}{\alpha^{2}\eta n}+\left|\E[\bar{X}_{n}-x^{\ast}]-\E_{\pi_{\eta}}[X-x^{\ast}]\right|\left(\left|\E[\bar{X}_{n}-x^{\ast}]\right|+\left|\E_{\pi_{\eta}}[X-x^{\ast}]\right|\right)\\
    &\le \frac{1}{n}\sum_{i=m+1}^{m+n}\sqrt{(1-\alpha\eta)^{i}2d/\alpha}\left(2\sqrt{2d/\alpha}\right)+\frac{2}{\alpha^{2}\eta n}\\
    &\quad+\frac{1}{n}\sum_{i=m+1}^{m+n}\left|\E[X_{i}]-\E_{\pi_{\eta}}[X]\right|\left(\frac{1}{n}\sum_{i=m+1}^{m+n}\E\left[\left|X_{i}-x^{\ast}\right|\right]+\left|\E_{\pi_{\eta}}[X-x^{\ast}]\right|\right)
    \\
    &\le \frac{1}{n}\sum_{i=m+1}^{m+n}\sqrt{(1-\alpha\eta)^{i}2d/\alpha}\left(2\sqrt{2d/\alpha}\right)+\frac{2}{\alpha^{2}\eta n}\\
    &\quad+\frac{1}{n}\sum_{i=m+1}^{m+n}\sqrt{(1-\alpha\eta)^{i}(2d/\alpha)}\left(2\sqrt{2d/\alpha}\right)\\
    &\le \frac{8d}{\alpha n}\sum_{i=m+1}^{m+n}(1-\alpha\eta)^{i/2}+\frac{2}{\alpha^{2}\eta n}
    \le \frac{8d}{\alpha n}\frac{(1-\alpha\eta)^{(m+1)/2}}{1-(1-\alpha\eta)^{1/2}}+\frac{2}{\alpha^{2}\eta n}\\
    &\le \frac{16d}{\alpha^{2}\eta n}(1-\alpha\eta)^{(m+1)/2}+\frac{2}{\alpha^{2}\eta n}
    \le \frac{\epsilon}{6}+\frac{\epsilon}{6}=\frac{\epsilon}{3},
\end{align*}
where we used $\sqrt{1-x}\le 1-x/2$ for $x\in[0,1]$ on the second last inequality and the condition on $n$ along with the inequalities
\begin{align*}
    &\frac{16d}{\alpha^{2}\eta n}(1-\alpha\eta)^{(m+1)/2}\le\frac{16\alpha^{3}\epsilon^{2}d}{2^{9}\cdot 3^{2}\alpha^{2}(9d+\log(4/\delta))}\le \frac{\alpha^{2}\epsilon d}{2^{5}3^{4}d}\le \frac{\epsilon}{6},\\
    &\frac{2}{\alpha^{2}\eta n}\le \frac{\alpha^{2} \epsilon}{2^8\cdot 3^2(9d+\log(4/\delta))}\le \frac{\epsilon}{2^8\cdot 3^4}\le \frac{\epsilon}{6}
\end{align*} 
in the last line (note that $\alpha\epsilon\le 1$ and obviously $d\ge 1$).

(Step 3) In the third place, we apply Theorem \ref{thm:concentration} to bound the variance term.
Noting that (i) the log-Sobolev constant of $\delta_{x^{\ast}}$ is $0$, and (ii) $\log((4d)/(\alpha\eta n))=\log(4(\E_{X_{0}\sim\delta_{x^{\ast}}}[|X_{0}-x^{\ast}|^{2}]+d/\alpha)/(\eta n))$, we see that the condition on $m$ of Theorem \ref{thm:concentration} is satisfied.
Since $(9d+4\log(4/\delta))/(\alpha\eta n)\le 1$ by the conditions on $n$ and $\epsilon\le 1/\alpha$, with probability at least $1-\delta$,
\begin{equation*}
    \left\|\hat{\Sigma}_{n}-\E\left[\hat{\Sigma}_{n}\right]\right\|\le \frac{16}{\alpha}\sqrt{\frac{2(9d+4\log(4/\delta)^{-1})}{\alpha\eta n}}\le \frac{\epsilon}{3}
\end{equation*}
Combining the bounds, we obtain the desired result.
\end{proof}

\subsection{Proof of Proposition \ref{prop:comp:ep}}
We first prepare a concentration bound for $\tilde{\Sigma}_{N}$ around its expectation.
\begin{proposition}\label{prop:concentration:ep}
    Suppose that Assumption \ref{asmp:potential} holds.
    If $\eta\in(0,1/\beta)$, for any $\delta>0$, $m\ge0$, and $N\ge1$, with probability at least $1-4\delta$,
    \begin{equation*}
        \left\|\tilde{\Sigma}_{N}-\E\left[\tilde{\Sigma}_{N}\right]\right\|\le \frac{16}{\alpha}\left(\sqrt{\frac{9d+4\log\delta^{-1}}{N}}\vee \frac{9d+4\log\delta^{-1}}{N}\right).
    \end{equation*}
\end{proposition}

We again consider centred samples $X_{m+1}^{k,\ctr}:=X_{m+1}^{k}-\E[X_{m+1}^{k}]$ for $k=1,\ldots,N$; note that $\tilde{\Sigma}_{N}$ can be rewritten as
\begin{equation*}
    \tilde{\Sigma}_{N} = \frac{1}{N}\sum_{k=1}^{N}X_{m+1}^{k,\ctr}(X_{m+1}^{k,\ctr})^{\top} - \left(\frac{1}{N}\sum_{k=1}^{N}X_{m+1}^{k,\ctr}\right)\left(\frac{1}{N}\sum_{k=1}^{N}X_{m+1}^{k,\ctr}\right)^{\top}.
\end{equation*}

\begin{proof}
    We first remark some facts used in this proof.
    Lemma \ref{lem:vem19} combined with the fact that $\delta_{x^{\ast}}$ satisfies an LSI($2/\alpha$) yields that the distribution of each $X_{m+1}^{k}$ satisfies an LSI($2/\alpha$).
    By tensorization \citep{bakry2014analysis}, the joint distribution $(X_{m+1}^{1},\ldots,X_{m+1}^{N})$ (as well as its translation $(X_{m+1}^{1,\ctr},\ldots,X_{m+1}^{N,\ctr})$) satisfies an LSI($2/\alpha$).

    (Step 1: Concentration for the sample mean)
    Using Proposition \ref{prop:lsi} and Lemma \ref{lem:nai25} with $A$ such that 
    \begin{equation*}
        A = \frac{1}{N^{2}}\left[\begin{matrix}
            uu^{\top} & 0 & \cdots & 0\\
            0 & uu^{\top} & \cdots & 0\\
            \vdots & \vdots & \ddots & \vdots\\
            0 & 0 & \cdots & uu^{\top}
        \end{matrix}\right]
    \end{equation*}
    (note that $\|A\|=1/N$), for any $u\in\bbS^{d-1}$ and $\lambda\le \alpha N/8$, we have
    \begin{equation*}
        \E\left[\exp\left(\lambda\left(\frac{1}{N}\sum_{k=1}^{N}u^{\top}X_{m+1}^{k,\ctr}\right)^{2}-\E\left[\left(\frac{1}{N}\sum_{k=1}^{N}u^{\top}X_{m+1}^{k,\ctr}\right)^{2}\right]\right)\right]\le \exp\left(\frac{8\lambda^{2}}{\alpha N}\cdot\frac{2}{\alpha N}\right),
    \end{equation*}
    where it follows from Lemma \ref{lem:pi} and the resulting bound such that
    \begin{equation*}
        \E\left[\left(\frac{1}{N}\sum_{k=1}^{N}u^{\top}X_{m+1}^{k,\ctr}\right)^{2}\right]\le \frac{2}{\alpha N}.
    \end{equation*}
    Then, using the net argument \citep{vershynin2018high}, for some $1/4$-net $\calN\subset\bbS^{d-1}$ with $|\calN|\le 9^{d}$, for any $\lambda\in[0,\alpha^{2}N/8]$ and $t\ge0$,
    \begin{align*}
        &\Pr\left(\left\|\left(\frac{1}{N}\sum_{k=1}^{N}X_{m+1}^{k,\ctr}\right)\left(\frac{1}{N}\sum_{k=1}^{N}X_{m+1}^{k,\ctr}\right)^{\top}-\E\left[\left(\frac{1}{N}\sum_{k=1}^{N}X_{m+1}^{k,\ctr}\right)\left(\frac{1}{N}\sum_{k=1}^{N}X_{m+1}^{k,\ctr}\right)^{\top}\right]\right\|\ge t\right)\\
        &\le \sum_{u\in\calN}\Pr\left(\left|\left(\frac{1}{N}\sum_{k=1}^{N}u^{\top}X_{m+1}^{k,\ctr}\right)^{2}-\E\left[\left(\frac{1}{N}\sum_{k=1}^{N}u^{\top}X_{m+1}^{k,\ctr}\right)^{2}\right]\right|\ge \frac{t}{2}\right)\\
        &\le 2\cdot 9^{d}\exp\left(\frac{16\lambda^{2}}{\alpha^{2} N^{2}}-\frac{\lambda t}{2}\right).
    \end{align*}
    Lemma \ref{lem:elem1} (with $a=16/(\alpha^{2}N^{2})$, $b=t/2$, $c=\alpha N/8$) yields that
    \begin{align*}
        &\Pr\left(\left\|\left(\frac{1}{N}\sum_{k=1}^{N}X_{m+1}^{k,\ctr}\right)\left(\frac{1}{N}\sum_{k=1}^{N}X_{m+1}^{k,\ctr}\right)^{\top}-\E\left[\left(\frac{1}{N}\sum_{k=1}^{N}X_{m+1}^{k,\ctr}\right)\left(\frac{1}{N}\sum_{k=1}^{N}X_{m+1}^{k,\ctr}\right)^{\top}\right]\right\|\ge t\right)\\
        &\le 2\exp\left(d\log 9-\left(\frac{\alpha^{2}N^{2}}{256}t^{2}\right)\wedge \left(\frac{\alpha N}{32}t\right)\right).
    \end{align*}
    Lemma \ref{lem:elem2} gives that for any $\delta\in(0,1)$, with probability at least $1-2\delta$,
    \begin{align*}
        &\left\|\left(\frac{1}{N}\sum_{k=1}^{N}X_{m+1}^{k,\ctr}\right)\left(\frac{1}{N}\sum_{k=1}^{N}X_{m+1}^{k,\ctr}\right)^{\top}-\E\left[\left(\frac{1}{N}\sum_{k=1}^{N}X_{m+1}^{k,\ctr}\right)\left(\frac{1}{N}\sum_{k=1}^{N}X_{m+1}^{k,\ctr}\right)^{\top}\right]\right\|&\\
        &\le \sqrt{\frac{256(d\log9+\log\delta^{-1})}{\alpha^{2}N^{2}}}\vee \left(\frac{32(d\log9+\log\delta^{-1})}{\alpha N}\right)=\frac{32(d\log9+\log\delta^{-1})}{\alpha N}.
    \end{align*}

    (Step 2: Concentration for the sample second moment matrix) We next consider the concentration for the sample second moment matrix.
    Using Proposition \ref{prop:lsi} and Lemma \ref{lem:nai25}, for any $u\in\bbS^{d-1}$ and $\lambda\le \alpha^{2}N/8$, we have
    \begin{equation*}
        \E\left[\exp\left(\frac{\lambda}{N}\sum_{k=1}^{N}\left((u^{\top}X_{m+1}^{k,\ctr})^{2}-\E[(u^{\top}X_{m+1}^{k,\ctr})^{2}]\right)\right)\right]\le \exp\left(\frac{8\lambda^{2}}{\alpha N}\cdot\frac{2}{\alpha}\right),
    \end{equation*}
    where it follows from Lemma \ref{lem:pi} and the resulting bound such that
    \begin{equation*}
        \frac{1}{N}\sum_{k=1}^{N}\E\left[(u^{\top}X_{m+1}^{k,\ctr})^{2}\right]=\E\left[(u^{\top}X_{m+1}^{1,\ctr})^{2}\right]\le \frac{2}{\alpha}.
    \end{equation*}
    Then, using the net argument \citep{vershynin2018high}, for some $1/4$-net $\calN\subset\bbS^{d-1}$ with $|\calN|\le 9^{d}$, for any $\lambda\in[0,\alpha N/8]$ and $t\ge0$,
    \begin{align*}
        &\Pr\left(\left\|\frac{1}{N}\sum_{k=1}^{N}X_{m+1}^{k,\ctr}(X_{m+1}^{k,\ctr})^{\top}-\E\left[\frac{1}{N}\sum_{k=1}^{N}X_{m+1}^{k,\ctr}(X_{m+1}^{k,\ctr})^{\top}\right]\right\|\ge t\right)\\
        &\le \sum_{u\in\calN}\Pr\left(\left|\frac{1}{N}\sum_{k=1}^{N}(u^{\top}X_{m+1}^{k,\ctr})^{2}-\E\left[\frac{1}{N}\sum_{k=1}^{N}(u^{\top}X_{m+1}^{k,\ctr})^{2}\right]\right|\ge \frac{t}{2}\right)\\
        &\le 2\cdot 9^{d}\exp\left(\frac{16\lambda^{2}}{\alpha^{2} N}-\frac{\lambda t}{2}\right).
    \end{align*}
    Lemma \ref{lem:elem1} (with $a=16/(\alpha^{2}N)$, $b=t/2$, $c=\alpha N/8$) yields that
    \begin{align*}
        &\Pr\left(\left\|\left(\frac{1}{N}\sum_{k=1}^{N}X_{m+1}^{k,\ctr}\right)\left(\frac{1}{N}\sum_{k=1}^{N}X_{m+1}^{k,\ctr}\right)^{\top}-\E\left[\left(\frac{1}{N}\sum_{k=1}^{N}X_{m+1}^{k,\ctr}\right)\left(\frac{1}{N}\sum_{k=1}^{N}X_{m+1}^{k,\ctr}\right)^{\top}\right]\right\|\ge t\right)\\
        &\le 2\exp\left(d\log 9-\left(\frac{\alpha^{2}N}{256}t^{2}\right)\wedge \left(\frac{\alpha N}{32}t\right)\right).
    \end{align*}
    Lemma \ref{lem:elem2} gives that for any $\delta\in(0,1)$, with probability at least $1-2\delta$,
    \begin{align*}
        &\left\|\left(\frac{1}{N}\sum_{k=1}^{N}X_{m+1}^{k,\ctr}\right)\left(\frac{1}{N}\sum_{k=1}^{N}X_{m+1}^{k,\ctr}\right)^{\top}-\E\left[\left(\frac{1}{N}\sum_{k=1}^{N}X_{m+1}^{k,\ctr}\right)\left(\frac{1}{N}\sum_{k=1}^{N}X_{m+1}^{k,\ctr}\right)^{\top}\right]\right\|&\\
        &\le \sqrt{\frac{256(d\log9+\log\delta^{-1})}{\alpha^{2}N}}\vee \left(\frac{32(d\log9+\log\delta^{-1})}{\alpha N}\right)\\
        &=\frac{1}{\alpha}\sqrt{\frac{256(d\log9+\log\delta^{-1})}{N}}\vee \left(\frac{32(d\log9+\log\delta^{-1})}{N}\right).
    \end{align*}

    (Step 3: Conclusion) Combining the results in Steps 1 and 2, we obtain 
    \begin{align*}
        \left\|\tilde{\Sigma}_{N}-\E\left[\tilde{\Sigma}_{N}\right]\right\|
        &\le \frac{2}{\alpha}\sqrt{\frac{64(9d+4\log\delta^{-1})}{N}}\vee \left(\frac{8(d\log9+\log\delta^{-1})}{N}\right)\\
        &\le \frac{16}{\alpha}\left(\sqrt{\frac{9d+4\log\delta^{-1}}{N}}\vee \frac{9d+4\log\delta^{-1}}{N}\right),
    \end{align*}
    which is the desired conclusion.
\end{proof}

We give the proof of Proposition \ref{prop:comp:ep}.
We introduce the notation for the sample mean $\tilde{X}_{N} = \frac{1}{N}\sum_{k=1}^{N}X_{m+1}^{k}$.
\begin{proof}[Proof of Proposition \ref{prop:comp:ep}]
    We again employ the bias--variance decomposition \eqref{eq:bvd} and estimate on $\|\Cov(\pi_{\eta})-\Cov(\pi)\|$ in the proof of Proposition \ref{prop:comp:ula}. 
    Without loss of generality, let $x^{\ast}=0\in\R^{d}$.
Without loss of generality, let $x^{\ast}=0\in\R^{d}$.
By Lemma \ref{lem:vem19} and tensorization \citep{bakry2014analysis}, the joint distribution $(X_{m+1}^{1},\ldots,X_{m+1}^{N})$ satisfies an LSI($2/\alpha$).
Then, the argument similar to that in the proof of Proposition \ref{prop:comp:ula} yields
\begin{align*}
    \left\|\E\left[\hat{\Sigma}_{N}\right]-\Cov(\pi_{\eta})\right\|
    &\le \left\|\E\left[\frac{1}{N}\sum_{k=1}^{N}X_{m+1}^{k}(X_{m+1}^{k})^{\top}\right]-\E_{\pi_{\eta}}\left[XX^{\top}\right]\right\|\\
    &\quad+\left\|\E\left[\left(\tilde{X}_{N}\right)\left(\tilde{X}_{N}\right)^{\top}\right]-\E_{\pi_{\eta}}\left[X\right]\E_{\pi_{\eta}}\left[X\right]^{\top}\right\|\\
    &\le \left\|\E\left[X_{m+1}^{1}(X_{m+1}^{1})^{\top}\right]-\E_{\pi_{\eta}}\left[XX^{\top}\right]\right\|\\
    &\quad+\left\|\E\left[\tilde{X}_{N}\left(\tilde{X}_{N}\right)^{\top}\right]-\E\left[\tilde{X}_{N}\right]\E\left[\tilde{X}_{N}\right]^{\top}\right\|\\
    &\quad+\left\|\E\left[\tilde{X}_{N}\right]\E\left[\tilde{X}_{N}\right]^{\top}-\E_{\pi_{\eta}}\left[X\right]\E_{\pi_{\eta}}\left[X\right]^{\top}\right\|\\
    &\le \frac{1}{N}\sum_{k=1}^{N}\sqrt{(1-\alpha\eta)^{m+1}2d/\alpha}\left(2\sqrt{2d/\alpha}\right)+\frac{2}{\alpha N}\\
    &\quad+\left\|\E\left[X_{m+1}^{1}\right]\E\left[X_{m+1}^{1}\right]^{\top}-\E_{\pi_{\eta}}\left[X\right]\E_{\pi_{\eta}}\left[X\right]^{\top}\right\|\\
    &\le 2\sqrt{(1-\alpha\eta)^{m+1}2d/\alpha}\left(2\sqrt{2d/\alpha}\right)+\frac{2}{\alpha N}\\
    &\le \frac{8d}{\alpha}(1-\alpha\eta)^{(m+1)/2}+\frac{2}{\alpha N}\le \frac{\epsilon}{3},
\end{align*}
where we repeatedly used the assumption that $X_{m+1}^{k}$ has the same distribution as $X_{m+1}^{1}$ for all $k=1,\ldots,N$.
We now consider the variance term.
By the assumption on $N$, $(9d+4\log(4/\delta)^{-1})/N\le 1$. 
Therefore, Proposition \ref{prop:concentration:ep} yields that with probability $1-\delta$,
\begin{align*}
        \left\|\tilde{\Sigma}_{N}-\E\left[\tilde{\Sigma}_{N}\right]\right\|\le \frac{16}{\alpha}\sqrt{\frac{9d+4\log(4/\delta)^{-1}}{N}}\le \frac{\epsilon}{3}.
\end{align*}
This is the desired conclusion.
\end{proof}

\bibliographystyle{apalike}
\bibliography{bibliography}

\begin{thebibliography}{}

\bibitem[Adamczak, 2015]{adamczak2015note}
Adamczak, R. (2015).
\newblock {A note on the Hanson-Wright inequality for random vectors with
  dependencies}.
\newblock {\em Electronic Communications in Probability}, 20:1--13.

\bibitem[Altschuler and Chewi, 2024]{altschuler2024faster}
Altschuler, J.~M. and Chewi, S. (2024).
\newblock Faster high-accuracy log-concave sampling via algorithmic warm
  starts.
\newblock {\em Journal of the ACM}, 71(3):1--55.

\bibitem[Altschuler and Talwar, 2022]{altschuler2022concentration}
Altschuler, J.~M. and Talwar, K. (2022).
\newblock {Concentration of the Langevin Algorithm's Stationary Distribution}.
\newblock {\em arXiv preprint arXiv:2212.12629v2}.

\bibitem[Bakry et~al., 2014]{bakry2014analysis}
Bakry, D., Gentil, I., and Ledoux, M. (2014).
\newblock {\em Analysis and Geometry of Markov Diffusion Operators}, volume
  348.
\newblock Springer Science \& Business Media.

\bibitem[Chewi, 2023]{chewi2023log}
Chewi, S. (2023).
\newblock {\em Log-Concave Sampling}.
\newblock Book draft. Available at \url{https://chewisinho.github.io}\EatDot.

\bibitem[Dalalyan, 2017]{dalalyan2017theoretical}
Dalalyan, A.~S. (2017).
\newblock Theoretical guarantees for approximate sampling from smooth and
  log-concave densities.
\newblock {\em Journal of the Royal Statistical Society: Series B (Statistical
  Methodology)}, 79(3):651--676.

\bibitem[Durmus and Moulines, 2017]{durmus2017nonasymptotic}
Durmus, A. and Moulines, E. (2017).
\newblock {Nonasymptotic convergence analysis for the unadjusted Langevin
  algorithm}.
\newblock {\em The Annals of Applied Probability}, 27(3):1551 -- 1587.

\bibitem[Durmus and Moulines, 2019]{durmus2019high}
Durmus, A. and Moulines, E. (2019).
\newblock {High-dimensional Bayesian inference via the unadjusted Langevin
  algorithm}.
\newblock {\em Bernoulli}, 25(4A):2854--2882.

\bibitem[Han and Li, 2020]{han2020moment}
Han, F. and Li, Y. (2020).
\newblock Moment bounds for large autocovariance matrices under dependence.
\newblock {\em Journal of Theoretical Probability}, 33(3):1445--1492.

\bibitem[Kook and Zhang, 2024]{kook2024covariance}
Kook, Y. and Zhang, M.~S. (2024).
\newblock {Covariance estimation using Markov chain Monte Carlo}.
\newblock {\em arXiv preprint arXiv:2410.17147}.

\bibitem[Nakakita et~al., 2025]{nakakita2025corrigendum}
Nakakita, S., Alquier, P., and Imaizumi, M. (2025).
\newblock {Corrigendum to “Dimension-free bounds for sums of dependent
  matrices and operators with heavy-tailed distributions”}.
\newblock {\em Electronic Journal of Statistics}, 19(2):3273--3291.

\bibitem[Neeman et~al., 2024]{neeman2024concentration}
Neeman, J., Shi, B., and Ward, R. (2024).
\newblock {Concentration inequalities for sums of Markov-dependent random
  matrices}.
\newblock {\em Information and Inference: A Journal of the IMA}, 13(4):iaae032.

\bibitem[Shen et~al., 2025]{shen2025high}
Shen, T., Su, Z., and Wang, X. (2025).
\newblock {High-dimensional normal approximations for sums of Langevin Markov
  chains}.
\newblock {\em arXiv preprint arXiv:2512.19496}.

\bibitem[Sherman, 2023]{sherman2023eigenstructure}
Sherman, P.~J. (2023).
\newblock {On the Eigenstructure of the AR (1) Covariance}.
\newblock In {\em 2023 IEEE Statistical Signal Processing Workshop (SSP)},
  pages 6--10. IEEE.

\bibitem[van Handel, 2016]{vanhandel2016probability}
van Handel, R. (2016).
\newblock {\em {Probability in High Dimension}}.
\newblock Lecture notes for APC 550 at Princeton University. Available at
  \url{https://web.math.princeton.edu/~rvan/APC550.pdf}\EatDot.

\bibitem[Vempala and Wibisono, 2019]{vempala2019rapid}
Vempala, S. and Wibisono, A. (2019).
\newblock {Rapid Convergence of the Unadjusted Langevin Algorithm: Isoperimetry
  Suffices}.
\newblock {\em Advances in Neural Information Processing Systems}, 32.

\bibitem[Vershynin, 2018]{vershynin2018high}
Vershynin, R. (2018).
\newblock {\em High-Dimensional Probability: An Introduction with Applications
  in Data Science}, volume~47.
\newblock Cambridge University Press.

\end{thebibliography}

\end{document}